\documentclass[11pt]{amsart}
\usepackage[english]{babel} 
\usepackage{amssymb}
\usepackage{bbm}
\usepackage{amsfonts}
\usepackage{latexsym}
\usepackage[all]{xy}
\usepackage{amscd}
\usepackage{comment}
\usepackage{fullpage}

\usepackage{amsthm} 
\usepackage{amsfonts} 
\usepackage{MnSymbol} 
\usepackage{setspace} 
\usepackage{amsmath} 
\usepackage[margin=1in]{geometry}
\usepackage{tikz}
\usepackage{graphicx} 
\graphicspath{ {images/} } 

\numberwithin{equation}{section}
\newtheorem{theorem}[equation]{Theorem}

\theoremstyle{definition}
\newtheorem{definition}[equation]{Definition}
\newtheorem{example}[equation]{Example}
\newtheorem{conjecture}[equation]{Conjecture}
\newtheorem{remark}[equation]{Remark}

\begin{document}
\title[Order dividing bijective function from non-cyclic]
{Order dividing bijective function from non-cyclic\\ to cyclic groups of same finite order}
\author{Austin Allen}
\email{tallen12010@live.com}
\author{Ashley Chen}
\email{ashley.j.chen@gmail.com}
\author{Jessica Ding}
\email{jess.ding.77@gmail.com}
\author{Piyush Shroff}
\email{piyushilashroff@gmail.com}
\address{Department of Mathematics, Texas State University,
San Marcos, Texas 78666, USA}
\date{December 21, 2016}

\begin{abstract}

In this article we give an order-dividing bijective function between cyclic and non cyclic groups of finite order. In particular, we prove that there exists a bijective function from $D_{2n}$ to $\mathbb{Z}_{2n}$ for any natural integer $n$; and from $\mathbb{Z}_p \times \mathbb{Z}_k$ to $\mathbb{Z}_{pk}$ when $p$ is an odd prime and $k$ is not a multiple of $p$. 
%

\end{abstract}
\maketitle
\begin{section}{Introduction}


The problem was proposed in The Kourovka Notebook No. 18 by I.M. Isaacs \cite{MK}. Frieder Ladisch proved it for solvable groups \cite{La}.
In \cite{AMS}, an article was published proving that the order of the element from the non-cyclic group was greater than or equal to the order of the element from the cyclic group that it was mapped to. 
\end{section}

\begin{section}{Preliminaries}
All the preliminary definitions and theorems can be found in any undergraduate Group Theory textbook. In particular, authors have referred \cite{DF}, \cite{F}, and \cite{G}. However, we recall following Theorem from \cite{DF}.

\begin{theorem}
Let $G$ be a cyclic group and $a$ an element in $G$ where $o(a) = n$.
Let $k \in \mathbb{N}$. Then
$o(a^k) = \frac{n}{gcd(n, k)}$. 
\end{theorem}


\end{section}

\begin{section}{Dihedral Groups}

\begin{definition}\cite{DF} 
The dihedral group $D_{2n}$ is defined as $$D_{2n}=\langle r, s | r^n = 1 = s^2, rs = sr^{-1} \rangle.$$
\end{definition}


\begin{theorem}
All elements of the form $sr^b$, $b \in \mathbb{Z}_{n}$, 
are of order $2$.
\end{theorem}

\begin{proof}
Let $sr^b$ be an element of the dihedral group of order $2n$.
It is clear that $sr^b \neq 1$, 
even if $b = 0$.
The next smallest positive integer to check is $2$.
$$(sr^b)^2 = (sr^b)(sr^b) \Rightarrow 
(sr^{b-1})rs(r^b).$$
Using the relation $rs = sr^{-1}$, we get
$$(sr^{b-1})sr^{-1}r(r^{b-1}) 
=(sr^{b-2})rs(r^{b-1})$$ 
Using it again results in 
$(sr^{b-2})sr^{-1}r(r^{b-2})$.
After a finite number of iterations, we finally get
$$srsr = ssr^{-1}r = s^2 = 1$$.
\end{proof}
\newpage



\begin{theorem} 

For any natural integer $n$, 
there exists a function from 
$D_{2n}$ to $\mathbb{Z}_{2n}$ 
defined as $f(s^ar^b) = ka + 2b$, 
where $a$ belongs to \{$0, 1$\},
$b$ belongs to $\mathbb{Z}_{n}$,
and $k$ is an odd integer,
such that the order of $s^ar^b$ 
divides the order of $f(s^ar^b)$.

\end{theorem}

\begin{proof}  

It is clear to see that the function is bijective. It only remains to prove that $o(s^a r^b)$ divides $o(ka+2b)$ where $k$ is odd.


Consider the case when $a = 0$. 
All elements in this subset would be of the form $s^0r^b = r^b$, 
where $b$ belongs to $\mathbb{Z}_{n}$.
Now $f(r^b) = 2b$. Thus
we want to show that the order of $r^b$ divides
the order of $2b$. 
By Theorem 2.1, 
$$o(2b) = o(1^{2b}) = \frac{2n}{gcd(2n, 2b)} = \frac{n}{gcd(n, b)}$$ 


Again by Theorem 2.1, 
$$o(r^b) = \frac{o(r)}{gcd(o(r), b)}$$  

$$o(r^b) = \frac{n}{gcd(n, b)}$$ 

Thus,
$o(r^b) = o(2b)$. 
Hence, $o(r^b)$ divides $o(f(r^b))$. 



Consider the case when $a = 1$.
By Theorem 3.2, any element of the form $sr^b$ has an order of $2$.
The corresponding output of each element can be expressed as $k + 2b$. 
By Theorem 2.1,
$$o(k + 2b) = o(1^{k + 2b}) = \frac{2n}{gcd(2n, k+2b)} = 2 \cdot \frac{n}{gcd(2n, k + 2b)}.$$ 

Note that since $k+2b$ is odd, $gcd(2n, k+2b)$ must be an odd integer.
Thus, $gcd(2n, k + 2b)$ must divide $n$, so 
$\frac{n}{gcd(2n, k + 2b)}$ is an integer. 
Hence,
$o(sr^b)$ divides $o(f(sr^b))$. 


\end{proof}

\begin{example}
Consider $f:D_6\rightarrow \mathbb{Z}_6$ defined by, 
$f(s^ar^b) = a + 2b$ and
$f(s^ar^b) = 5a + 2b$.

\begin{table}[h]
    \centering
    \begin{tabular}{| c | c | c | c |}
    \hline
    Order of $D_6$ & $D_6$ & $\mathbb Z_6$ & Order of $\mathbb Z_6$ \\
    \hline
    $1$ & $1$ & $0$ & $1$ \\
    $3$ & $r$ & $2$ & $3$ \\
    $3$ & $r^2$ & $4$ & $3$ \\
    $2$ & $s$ & $1$ & $6$ \\
    $2$ & $sr$ & $3$ & $2$ \\
    $2$ & $sr^2$ & $5$ & $6$ \\
    \hline
\end{tabular}
\caption{Map from $D_6 \rightarrow \mathbb Z_6$ ($f(s^ar^b) = a + 2b$)}
\end{table}

\begin{table}[h]
    \centering
    \begin{tabular}{| c | c | c | c |}
    \hline
    Order of $D_6$ & $D_6$ & $\mathbb Z_6$ & Order of $\mathbb Z_6$ \\
    \hline
    $1$ & $1$ & $0$ & $1$ \\
    $3$ & $r$ & $2$ & $3$ \\
    $3$ & $r^2$ & $4$ & $3$ \\
    $2$ & $s$ & $5$ & $6$ \\
    $2$ & $sr$ & $1$ & $6$ \\
    $2$ & $sr^2$ & $3$ & $2$ \\
    \hline
\end{tabular}
\caption{Map from $D_6 \rightarrow \mathbb Z_6$ ($f(s^ar^b) = 5a + 2b$)}
\end{table}

\end{example}


\begin{conjecture} 
Given that $f(s^ar^b) = xa + yb$ is a order dividing bijective function 
where $x, y \in \mathbb{Z}$, then
$f(s^ar^b) = ya + xb$ is not a order dividing bijective function.
\end{conjecture}
\end{section}




\begin{section}{Direct Product Groups} 
\begin{definition}\cite{DF}
The direct product $G\times H$ of the groups $G, H$ with operation $\ast$, is the ordered pairs $(g,h)$ where $g\in G$ and $h\in H$ with operation defined componentwise:
$$(g_1,h_1)\ast (g_2,h_2) = (g_1\ast g_2, h_1\ast h_2).$$
\end{definition}
\noindent Here we restrict to the group $\mathbb{Z}_n \times \mathbb{Z}_m$.
\begin{theorem}\cite{F}
If $gcd(m, n) = 1$, then $\mathbb{Z}_n \times \mathbb{Z}_m$ 
is cyclic and isomorphic to $\mathbb{Z}_{mn}$,
and $(1, 1)$ is a generator of $\mathbb{Z}_n \times \mathbb{Z}_m$. 
\end{theorem}


The structures of the bijective functions we've explored
are exactly the same as the ones we used in the dihedral groups. 
The essential problem is to make sure the order of the inputs 
divide the order of the outputs. Note that order of any element $(a,b)$ in $\mathbb{Z}_n \times \mathbb{Z}_m$ is lcm of $o(a)$ and $o(b)$.

\begin{theorem}
For any odd prime $p$ and natural number $k$
such that $gcd( p, k ) = 1$, 
a bijective function $f$ whose domain is  
$\mathbb{Z}_p \times \mathbb{Z}_{kp}$ and
range is $\mathbb{Z}_{kp^2}$ 
can be defined as
$f( (a, b) ) = ka + pb$, where
$a \in \mathbb{Z}_p$ and $b \in \mathbb{Z}_{kp}$.
\end{theorem}

\begin{proof}
We split the proof into two cases.\\
{\bf Case I}: $a = 0$\\
Consider the domain, $\mathbb{Z}_p \times \mathbb{Z}_{kp}$. 
The elements in the domain are of the form $(0,b)$.
The order of $(0,b)$ is same as
$o(b)$ in $\mathbb{Z}_{kp}$. 
By Theorem 2.1,
$$o(b) = o(1^b) = \frac{kp}{gcd(b,kp)}.$$

The corresponding elements in the co-domain
are of the form $k \ast 0 + pb = pb$. 
Again, by\\ Theorem 2.1  order of $pb$ in $\mathbb{Z}_{kp^2}$ is
$$o(pb) = o(1^{pb}) = \frac{kp^2}{gcd(pb, kp^2)}.$$
Since $gcd(pb, kp^2) = p \cdot gcd(b, kp)$, we get
$$o(pb) = \frac{kp}{gcd(b, kp)}.$$
Hence, order of $(a,b)$ divides order of $ka+pb$.\\\\
{\bf Case II}: $a \neq 0$\\
Since $a \neq 0\in \mathbb{Z}_{p}$ , $gcd(a, p) = 1$.\\
Now consider the order of the element $(a, b)$ in 
$\mathbb{Z}_p \times \mathbb{Z}_{kp}$.
Since $gcd(a, p) = 1$, $p$ is the least positive number
such that $a^p \equiv 0 \mod p$. 
Thus the order of $(a, b)$ has to be a multiple of $p$.
Now we have to find the least positive integer $r$ 
such that $b^{pr} \equiv 0 \mod kp$. 

If $b$ is a multiple of $p$, then 
finding the order of $b$ in $\mathbb{Z}_{kp}$ 
is equivalent to finding the order of $b$ in $\mathbb{Z}_{k}$.
Then by Theorem 2.1, the order of $b$ is the same as
$$o(1^b) = \frac{k}{gcd(b, k)}.$$ 
Since the order of $(a, b)$ in $\mathbb{Z}_p \times \mathbb{Z}_{kp}$
has to be a multiple of $p$, and  
$\frac{k}{gcd(b, k)}$ and $p$ are relatively prime,
the order of $(a, b)$ is $p \cdot \frac{k}{gcd(b, k)}$.

If $b$ is not a multiple of $p$, then 
by Theorem 2.1
the order of $b$ in $\mathbb{Z}_{kp}$ is
$$o(1^b) = \frac{pk}{gcd(b, pk)}.$$
Since we assume that $b$ is not a multiple of $p$,
$b$ must be relatively prime to $p$.
Thus, $gcd(b, pk) = gcd(b, k)$ and hence,
$$o(b) = \frac{pk}{gcd(b, k)} = p \cdot \frac{k}{gcd(b, k)}.$$
Since $o(a, b) = lcm( o(a), o(b))$,
the order of $(a, b)$ in $\mathbb{Z}_p \times \mathbb{Z}_{kp}$
is $$p \cdot \frac{k}{gcd(b, k)}.$$
Now consider the corresponding output, $ka + pb$. 
We know that 
$$o(ka + pb) = o(1^{ka + pb})$$
so by Theorem 2.1,
$$o(1^{ka + pb}) = \frac{kp^2}{gcd(kp^2, ka + pb)}.$$ 
Since $gcd(a, p) = 1$ and $gcd(k, p) = 1$,
we know that $gcd(ka, p) = 1$. 
Thus,
$$gcd(kp^2, ka + pb) = gcd(k, ka + pb) = gcd(k, pb) = gcd(k, b)$$ 
Therefore, the order of $ka + pb$ in $\mathbb{Z}_{kp^2}$ is
$$o(ka+pb) = \frac{kp^2}{gcd(b, k)} = p^2\cdot\frac{k}{gcd(b, k)}.$$
Hence, the order divides.

\end{proof}

\begin{remark}

This theorem can in fact be extended to include more bijective functions 
by adding a coefficient $m$ to the product $ka$ in $ka + pb$ to get $mka + pb$.
The only restriction that needs to be added is that $gcd(m, p) = 1$. 

\end{remark}

\noindent{\bf Acknowledgement:} This paper is the part of the project submitted for 2016 Siemens competition and was partially funded by Texas State Mathworks. The authors would like to thank Texas State Mathworks Honors Summer Math Camp.
\end{section}

\quad

\quad


\begin{thebibliography}{KKZ}

\bibitem{AMS} The American Mathematical Monthly, Vol 109, No. 3 (March 2002), p. 299.

\bibitem{AA} H.\ Amiri, S.\ M.\ Jafarian Amiri, \textit{Sum of Element Orders on Finite Groups of the Same Order}, Journal of Algebra and its Applications, World Scientific Publishing Company, 2009.

\bibitem{C} Arthur Cayley, \textit{On the theory of groups as depending on the symbolic equation $\theta^n = 1$}, Philosophical Magazine, 1854.

\bibitem{DF} David S.\ Dummit and Richard M.\ Foote, \textit{Abstract Algebra}, John Wiley Sons, 2nd edition, 2004

\bibitem{F} John B.\ Fraleigh, \textit{A First Course in Abstract Algebra}, 7th edition.

\bibitem{G} Joseph A.\ Gallian, \textit{Contemporary Abstract Algebra}, Houghton Mifflin Harcourt, 4th edition, 1998.

\bibitem{La} Frieder Ladisch, \textit{``This is only a partial answer..."}, MathOverflow, 2012. 

\bibitem{MK} V.\ D.\ Mazurov, E.\ I.\ Khukhro, \textit{Unsolved Problems in Group Theory. The Kourovka Notebook. No. 18 (English version)}, American Mathematical Society, 18th edition, 2014.

\end{thebibliography}
\end{document}